\documentclass[12pt]{amsart}

\newcommand{\fA}{\mathfrak A}
\newcommand{\fD}{\mathfrak D}
\newcommand{\chim}{\Chi_{\bar M}}
\newcommand{\spwhat}{\;\what{ }\;}

\usepackage{verbatim}
\usepackage{amssymb}
\usepackage{upref}
\usepackage[all]{xy}
\usepackage{color}

\allowdisplaybreaks

\newtheorem{thm}{Theorem}[section]
\newtheorem*{thm*}{Theorem}
\newtheorem{lem}[thm]{Lemma}
\newtheorem{cor}[thm]{Corollary}
\newtheorem{prop}[thm]{Proposition}
\newtheorem{obs}[thm]{Observation}

\theoremstyle{definition}
\newtheorem{defn}[thm]{Definition}

\newtheorem{notn}[thm]{Notation}
\newtheorem{hyp}[thm]{Hypothesis}
\newtheorem*{notn*}{Notation}

\newtheorem*{hyp*}{Hypothesis}

\newtheorem{rem}[thm]{Remark}
\newtheorem*{rem*}{Remark}

\numberwithin{equation}{section}

\newcommand{\secref}[1]{Section~\textup{\ref{#1}}}

\newcommand{\thmref}[1]{Theorem~\textup{\ref{#1}}}
\newcommand{\corref}[1]{Corollary~\textup{\ref{#1}}}
\newcommand{\lemref}[1]{Lemma~\textup{\ref{#1}}}
\newcommand{\propref}[1]{Proposition~\textup{\ref{#1}}}

\newcommand{\hypref}[1]{Hypothesis~\textup{\ref{#1}}}

\newcommand{\midtext}[1]{\quad\text{#1}\quad}
\newcommand{\righttext}[1]{\quad\text{#1 }}
\renewcommand{\and}{\midtext{and}}
\renewcommand{\for}{\righttext{for}}
\newcommand{\all}{\righttext{for all}}

\renewcommand{\)}{\textup)}
\newcommand{\ie}{\emph{i.e.}}

\newcommand{\cf}{\emph{cf.}}
\newcommand{\etc}{\emph{etc.}}

\newcommand{\N}{\mathbb N}
\newcommand{\Z}{\mathbb Z}
\newcommand{\Q}{\mathbb Q}

\newcommand{\BB}{\mathcal B}
\newcommand{\PP}{\mathcal P}

\newcommand{\RR}{\mathcal R}

\newcommand{\MM}{\mathcal M}

\newcommand{\UU}{\mathcal U}

\renewcommand{\AA}{\mathcal A}
\renewcommand{\SS}{\mathcal S}

\newcommand{\Nb}{\bar N}

\newcommand{\Chi}{\raisebox{2pt}{\ensuremath{\chi}}}
\renewcommand{\epsilon}{\varepsilon}
\renewcommand{\a}{\alpha}

\DeclareMathOperator{\aut}{Aut}

\DeclareMathOperator{\supp}{supp}

\DeclareMathOperator*{\spn}{span}
\DeclareMathOperator*{\clspn}{\overline{\spn}}
\DeclareMathOperator*{\invlim}{\varprojlim}

\newcommand{\case}{& \text{if }}

\newcommand{\<}{\langle}
\renewcommand{\>}{\rangle}
\renewcommand{\iff}{\ensuremath{\Leftrightarrow}}
\newcommand{\minus}{\setminus}
\newcommand{\inv}{^{-1}}
\renewcommand{\bar}{\overline}
\newcommand{\what}{\widehat}
\newcommand{\wilde}{\widetilde}

\begin{document}
\title[Cuntz-Li algebras]{A crossed-product approach to the Cuntz-Li algebras}

\author[Quigg]{John Quigg}
\address{School of Mathematical and Statistical Sciences
\\Arizona State University
\\Tempe, Arizona 85287}
\email{quigg@asu.edu}

\author[Kaliszewski]{S. Kaliszewski}
\address{School of Mathematical and Statistical Sciences
\\Arizona State University
\\Tempe, Arizona 85287}
\email{kaliszewski@asu.edu}

\author[Landstad]{Magnus~B. Landstad}
\address{Department of Mathematical Sciences\\
Norwegian University of Science and Technology\\
NO-7491 Trondheim, Norway}
\email{magnusla@math.ntnu.no}



\begin{abstract}
Cuntz and Li have defined a $C^*$-algebra associated to any integral domain, using generators and relations, and proved that it is simple and purely infinite and that it is stably isomorphic to a crossed product of a commutative $C^*$-algebra. We give an approach to a class of $C^*$-algebras  containing those studied by Cuntz and Li, using the general theory of $C^*$-dynamical systems associated to certain semidirect product groups. Even for the special case of the Cuntz-Li algebras, our development is new.
\end{abstract}
\maketitle

\section{Introduction}
\label{intro}

For an integral domain $R$, in \cite{CLintegral} Cuntz and Li define a remarkable $C^*$-algebra $\fA[R]$ via generators and relations.
Using only the relations, they show that $\fA[R]$ is simple and purely infinite.
They then show that $\fA[R]$ is isomorphic to a corner of
the crossed product of a commutative $C^*$-algebra by an action of the $ax+b$ group of the quotient field $Q(R)$.

Our purpose here is to give an alternative approach to $\fA[R]$, essentially the reverse of that in \cite{CLintegral}: we \emph{begin} with a crossed product,
which we show is simple and purely infinite using the theory of $C^*$-dynamical systems.
Then we show that a certain corner of this crossed product
is generated by elements satisfying relations similar to those of Cuntz and Li, 
and that this corner $C^*$-algebra is universal for these relations.
It follows that this corner is in fact isomorphic to the algebra $\fA[R]$ of \cite{CLintegral}.
Of course, this isomorphism is not surprising, since $\fA[R]$ is simple by the results of \cite{CLintegral}; the point is, though, that we do not \emph{assume} the simplicity result of \cite{CLintegral}, rather we \emph{deduce} it as a consequence of our results.
The embedding of the Bost-Connes algebra in  $\fA[R]$ is now easy to explain, since it is a corner of a subalgebra of the crossed product obtained by cutting down with the  same projection.

Other approaches to the Cuntz-Li algebra are given in 
\cite{BoavaExelDomains, BHLRboundary, CLintegral2, DuncanDomains, LarsenLi, LarsenLi2, LiRing, Yamashita}.

We actually do most of our work in a somewhat more general context than \cite{CLintegral}:
we use a semidirect product that incorporates the essential features of the $ax+b$ group of the integral domain $R$, but there is no ring.
More precisely, we have a semidirect product $G=N\rtimes H$ satisfying certain mild conditions regarding a certain normal subgroup $M$ of $N$.
We make heavy use of the completion $\bar N$ relative to the subgroup topology generated by the conjugates of $M$.

In \secref{G} we use standard $C^*$-crossed-product theory,
specifically results of Archbold, Laca, and Spielberg,
to prove that,
if $D$ denotes the $C^*$-algebra generated by the characteristic functions of the cosets in $N$ of the conjugates of $M$,
then $D=C_0(\bar N)$ and the reduced crossed product $D\rtimes_r G$ is simple and purely infinite.

In \secref{universal-sec} we show that the full crossed product $D\rtimes G$ is the enveloping $C^*$-algebra of an algebraic crossed product $D_0\rtimes G$, which in turn is universal for appropriate covariant representations of $(D_0,G)$.
Assuming that the action of $G$ on $D$ is ``regular'' in the sense that $D\rtimes G=D\rtimes_r G$, we prove that $D_0\rtimes G$ has a unique $C^*$-norm,
and consequently the corner $p(D\rtimes G)p$ is the enveloping $C^*$-algebra of $p(D_0\rtimes G)p$, where $p$ is the characteristic function of $M$.

In \secref{hecke} we show that $M$ is a Hecke subgroup of $G$, and the associated Hecke algebra $\MM$ has a universal $C^*$-algebra $C^*(\MM)$, which in turn embeds as a full corner of $C^*(\bar G)$ (where $\bar G$ is the completion relative to the subgroup topology.
When $N$ is appropriately self-dual, we conclude that $C^*(\MM)$ embeds faithfully in the corner $p(D\rtimes G)p$.

In \secref{integral} we specialize to the case where $G$ is the $ax+b$ group of the field of fractions
of an integral domain $R$.  In this case $G$ is amenable, so
the full crossed product $D\rtimes G \cong D\rtimes_r G$ is simple and purely infinite.

In \secref{BC} we specialize even further, as Cuntz and Li do, 
so that $R$ is the ring of integers of an algebraic number field 
$K=\Q(\theta)$. If $R=\Z[\theta]$, then $C^*(\MM)$ is the generalized Bost-Connes $C^*$-algebra. 
Anticipating the isomorphism $p(D\rtimes G)p\cong \fA[R]$ of \secref{ring},
the results of \secref{hecke} allow us to
recover the embedding of Cuntz and Li of the generalized Bost-Connes algebra into $\fA[R]$.

In \secref{further} we give an example in a different direction from the Cuntz-Li work, namely with $G$ a lamplighter group.

Cuntz and Li define their algebras using generators and relations. 
In \secref{gnr} we show that the corners $p(D\rtimes G)p$ and $p(D\rtimes H)p$ are universal $C^*$-algebras
for similar generators and relations.
We use standard dilation techniques as presented by Douglas. 

In \secref{ring} we use the universal property of \secref{gnr} to see that the corner
$p(D\rtimes G)p$ is isomorphic to the ring $C^*$-algebra $\fA[R]$ of Cuntz and Li.  
We thus recover their result that $\fA[R]$ is simple and purely infinite.

We are grateful to the referee for several helpful comments that have substantially improved the paper.


\section{Preliminaries}
\label{prelim}

Throughout this paper, 
we let $G=N\rtimes H$ be a (discrete) semidirect product group 
with normal subgroup $N$ and quotient group $H$, 
and we let $M{\neq \{e\}}$ be a normal subgroup of $N$.
We assume that the family
\[
\UU:=\{aMa\inv : a\in H\}
\]
of normal subgroups of $N$
is a \emph{filter base} in the sense that 
\begin{equation}
\label{filter base}
\text{for all $U,V\in\UU$ there exists $W\in\UU$ such that {$W\subseteq U\cap V$}},
\end{equation}
and that $\UU$ is \emph{separating} in the sense that
\begin{equation}
\label{separating}
\bigcap_{U\in\UU} U=\{e\}.
\end{equation}
Because we require $M\neq\{e\}$,
condition~\eqref{separating} implies that  $M\neq N$ and $H\neq\{e\}$.

We will at times also explicitly make one or both
of the following assumptions:
that $\UU$ has \emph{finite quotients} in the sense that
\begin{equation}\label{finite}
 \bigl|U/V\bigr|<\infty\righttext{whenever {$V\subseteq U$} in $\UU$,}
\end{equation}
and that $H$ acts \emph{effectively} on $M$, so that
\begin{equation}\label{effective}
\text{for each $a\neq e$ in $H$ there exists $s\in M$ such that $asa\inv\ne s$.}
\end{equation}

The following results are essentially \cite[Chapter~III, \S7.3, Corollary~2]{bourbaki:GTI}.

\begin{lem}
\label{completion}
Let $N$ be as above with the separating filter base $\UU$ of normal subgroups. Then the inverse limit
\[
\invlim_{U\in\UU}N/U
\]
is a Hausdorff completion of $N$.
\end{lem}

\begin{lem}
\label{N/U}
For each $U\in\UU$, the map
$xU\mapsto x\bar U$
gives a group isomorphism of $N/U$ onto $\Nb/\bar U$. 
\end{lem}

\begin{lem}
\label{locally compact}
The completion $\Nb$ is locally compact if and only if there exists $U\in\UU$
such that
\[
\bigl|U/V\bigr|<\infty\righttext{for all $V\in\UU$ with $V\subseteq U$}.
\]
\end{lem}
\begin{proof}
First assume that 
we have $U\in\UU$ such that $|U/V|<\infty$ whenever $V\in\UU$ and $V\subseteq U$, 
and put
\[
\UU'=\{V\in\UU:V\subseteq U\}.
\]
Since $\UU'$ is cofinal in $\UU$, we have a natural isomorphism
\[
\Nb\cong \invlim_{V\in\UU'}N/V.
\]
The closure $\bar U$ is a compact (and open) subgroup of the completion $\Nb$, because it is an inverse limit of the finite groups $\{U/V:V\in\UU'\}$. It follows that $\Nb$ is locally compact, because it contains a compact open subgroup.

Conversely, assume that $\Nb$ is locally compact, and choose a compact neighborhood $W$ of $e$. Since $\{\bar U:U\in\UU\}$ is a local base of $\Nb$, there exists $U\in\UU$ such that $\bar U\subseteq W$. Thus $\bar U$ is compact. Let $V\in\UU$ with $V\subseteq U$. Then $\bar V$ is an open normal subgroup of $\bar U$, so the quotient group $\bar U/\bar V$ is compact and discrete, and hence is finite.
Since $U/V\cong \bar U/\bar V$ by \lemref{N/U}, we have $|U/V|<\infty$.
\end{proof}

\section{The $ax+b$ group action}
\label{G}

We define the following ``$ax+b$''-type action $\beta$ of $G$ on $N$:
for $g\in G$, let $\beta_g\colon N\to N$ by
\[
\beta_g(y)=xaya\inv
\]
whenever $g=xa$ with $x\in N$ and $a\in H$;
thus $N\subseteq G$ acts on $N$ by left translation and $H\subseteq G$ acts by conjugation.
It is not hard to check that $\beta_g$ is a well-defined bijection of $N$ onto itself
(note that $\beta_g$ is a group homomorphism of~$N$ if and only if $g\in H$),
and that $g\mapsto \beta_g$ is a group homomorphism.  
In fact, $\beta$ corresponds to the left action of $G$ on $G/H$ under the natural identification
of $G/H = (NH)/H$ with $N$. 

Note that if $g=xa$ as above, for $y\in N$ and $U\in\UU$ we have
\begin{equation}\label{U}
\beta_g(yU)=(xaya\inv)(aUa\inv).
\end{equation}
Thus the set
\[
\PP=\{\Chi_{xU} : x\in N, U\in\UU\} \subseteq \ell^\infty(N)
\]
is invariant under the action
$\alpha$ of $G$ on $\ell^\infty(N)$ associated to $\beta$ 
in the usual way by
\[
\alpha_g(f)=f\circ\beta_{g\inv}.
\]
If we further define
\[
 D = C^*(\PP)\subseteq \ell^\infty(N),
\]
it follows that $\alpha$ restricts to an action, 
still denoted by $\alpha$,
of $G$ on the $C^*$-algebra $D$.

\begin{thm}
\label{simple}
With the above notation and assumptions, suppose that $\UU$ has finite quotients \(condition~\eqref{finite}\)
and that $H$ acts effectively on $M$ \(condition~\eqref{effective}\).  
Then the reduced crossed product $D\rtimes_{\alpha,r} G$ is simple and purely infinite.
\end{thm}

The first step of the proof is to identify the commutative $C^*$-algebra~$D$.
Notice that 
with $g=xa$ as at~\eqref{U}, 
$\beta_g$ induces a bijection of $\UU$ onto $\{ xU : U\in \UU \}$;
thus $\beta_g$ is a uniform isomorphism of $N$ (with the subgroup topology determined by $\UU$)
onto itself. 
It follows that the action $\beta$ of $G$ on $N$ extends 
uniquely to an action $\bar\beta$ of $G$ by homeomorphisms of the completion $\Nb$.
We let $\bar\alpha$ denote the associated 
action of $G$ on $C_0(\Nb)$, so that
\[
\bar\alpha_g(f)=f\circ\bar\beta_{g\inv}\quad\text{ for $f\in C_0(\Nb)$ and $g\in G$.}
\]

The following result is elementary, and we claim no originality; 
however, we could not find it in the literature, so we include the proof for completeness.

\begin{prop}
\label{extend}
If $\UU$ has finite quotients, then
the restriction map $\rho$ of $C_0(\Nb)$ into $\ell^\infty(N)$ given by 
\[
\rho(f) =  f|_N
\]
is an $\bar\alpha - \alpha$ equivariant isomorphism of  $C_0(\Nb)$ onto $D$.
\end{prop}

\begin{proof}
Since $N$ is dense in $\Nb$, 
$\rho$ gives an isometric homomorphism of $C_0(\Nb)$ into $\ell^\infty(N)$.
The cosets $\{x\bar U:x\in \Nb,U\in\UU\}$ form a base for the topology of $\Nb$ consisting of compact sets
(see the proof of Lemma~\ref{locally compact}), 
so $C_0(\Nb)$ is generated by the set $\SS = \{\Chi_{x\bar U} : x\in \Nb, U\in\UU\}$
of characteristic functions.
Again since $N$ is dense, for each $x\in\Nb$ and $U\in\UU$ there exists $y\in N$ such that $x\bar U=y\bar U$.
Since $(\Chi_{y\bar U})|_N = \Chi_{yU}$ for such $y$ and $U$,
we have
\[
 \rho(\SS)  
= \{ \Chi_{yU} : y\in N, U\in\UU \} = \PP,
\]
and it follows that $\rho$ maps $C_0(\Nb)$ onto $D$. 

For the equivariance, since each $\bar\beta_g$ is the extension to $\Nb$ of $\beta_g$,
for $f\in C_0(\Nb)$ we have
\[
 \rho(\bar\alpha_g(f)) = (f\circ\bar\beta_{g\inv})|_N
= (f|_N)\circ\beta_{g\inv} = \alpha_g(\rho(f)).
\qedhere
\]
\end{proof}

Recall from \cite{AStopfree, LSboundary} that an action $\gamma$ of a discrete group $G$
on a locally compact Hausdorff space $X$ is:
\begin{enumerate}
\item \emph{minimal} if for every $x\in X$ the orbit $\{\gamma_g(x):g\in G\}$ is dense;

\item \emph{locally contractive} if for every nonempty open set $O\subseteq X$ there exist $g\in G$ and a nonempty open set $O'\subseteq O$ such that
\[
\bar{\gamma_g(O')}\subsetneq O';
\]

\item \emph{topologically free} if for every $g\in G\minus\{e\}$ the set
\[
\{x\in X:\gamma_g(x)=x\}
\]
of fixed points has empty interior.
\end{enumerate}
In \cite{AStopfree} the term ``local boundary action'' is used instead of ``locally contractive action''.

\begin{lem}
\label{minimal}
If $\UU$ has finite quotients, then $\Nb$ is locally compact
and the action $\bar\beta$ of $G$ on $\Nb$ is minimal and locally contractive.
\end{lem}

\begin{proof}
Local compactness follows from Lemma~\ref{locally compact}.
For  minimality, let $x\in \Nb$.
Then the orbit of $x$ under the action of $G$ contains the coset $Nx$,
which is dense in $\Nb x=\Nb$.

For local contractivity,
since the cosets $\{y\bar U:y\in N,U\in\UU\}$ form a base for the topology of $\Nb$ consisting of closed sets,
it suffices to show that for every $y\in N$ and $U\in\UU$ there exists $g\in G$ such that
\[
\bar\beta_g(y\bar U)\subsetneq y\bar U.
\]
For this, first note that since $M\neq\{e\}$,
by~\eqref{separating} there exists $d\in H$ 
such that $dMd\inv \neq M$,
and then by~\eqref{filter base} we can choose $c\in H$
such that $cMc\inv \subseteq M\cap dMd\inv \subsetneq M$.
 
Now fix $b\in H$, so that $U=bMb\inv$ is an arbitrary element of $\UU$.
Setting $a=bcb\inv\in H$ gives
\[
 aUa\inv = b(cMc\inv)b\inv \subsetneq bMb\inv = U,
\]
and since the map $U\mapsto \bar U$ is injective on $\UU$
(Lemma~\ref{N/U}),
it follows that $\overline{aUa\inv} \subsetneq \bar U$.
Thus, for any $y\in N$,
if we further set $x = y(ay\inv a\inv) \in N$
and then $g=xa$, we have
\[
\bar\beta_g(y\bar U) 
= \overline{\beta_g(yU)} 
= \overline{(xaya\inv)(aUa\inv)}
= y\overline{aUa\inv} \subsetneq y\bar U,
\]
as desired.
\end{proof}

\begin{lem}
\label{free}
If 
$H$ acts effectively on $M$,
then the action $\bar\beta$ of $G$ on $\Nb$ is topologically free.
\end{lem}

\begin{proof}
Let $g\in G\minus\{e\}$. We must show that the fixed-point set
\[
\{y\in \Nb:\bar\beta_g(y)=y\}
\]
has empty interior. Suppose not.
Since $N$ is dense in $\bar N$
there is a nonempty open set $O\subseteq N$ which is fixed pointwise by $\beta_g$. Let $g=xa$ with $x\in N$ and $a\in H$.
Now, for any $y,z\in O$ we have
\[
xaya\inv=y\and xaza\inv=z,
\]
so
\[
x=yay\inv a\inv=zaz\inv a\inv,
\]
and hence
\[
y\inv z=ay\inv a\inv aza\inv=ay\inv za\inv,
\]
i.e., the open neighborhood $O\inv O$ of $e$ in $N$ consists of fixed points for
conjugation by the element $a$ of $H$.
By definition of the subgroup topology, we have 
$bMb\inv\subseteq O\inv O$ for some $b\in H$. Thus for all $s\in M$ we have
\[
absb\inv a\inv=bsb\inv,
\]
so $b\inv ab$ acts trivially by conjugation on $M$. 
If $H$ acts effectively on $M$ we must then have $b\inv ab=e$, so $a=e$, 
which in turn forces $xy=y$ for all $y\in O$, so $x=e$, and hence $g=e$, giving a contradiction.
\end{proof}

\begin{proof}[Proof of Theorem~\ref{simple}]
Since $D\rtimes_{\alpha,r}G\cong C_0(\Nb)\rtimes_{\bar\alpha,r}G$ by
Proposition~\ref{extend},
it suffices to show that $C_0(\Nb)\rtimes_{\bar\alpha,r}G$ is simple and
purely infinite.
Simplicity follows from \cite[Corollary of Theorem~1]{AStopfree} and
\cite[Theorem~2]{AStopfree},
since $\bar\beta$ is minimal by Lemma~\ref{minimal} and 
topologically free by Lemma~\ref{free}.
The crossed product is purely infinite by \cite[Theorem~9]{LSboundary},
since $\bar\beta$ is 
locally contractive by Lemma~\ref{minimal}.
\end{proof}

\begin{rem}
Note that in \thmref{simple} the action of $G$ on $\what D=\bar N$ is not free: if $g\in H$ then $\beta_g(e)=e$.
\end{rem}


\section{Universal $C^*$-algebras}
\label{universal-sec}

In this section we show that $D\rtimes_\alpha G$ is the enveloping $C^*$-algebra of an ``algebraic crossed product'' (\corref{enveloping crossed product}). The main result of this section is that under an extra \hypref{regular} on the action (namely that the full and reduced crossed products coincide), a certain corner of the algebraic crossed product has a unique $C^*$-norm, and moreover the corresponding corner of $D\rtimes_\alpha G$ is its enveloping $C^*$-algebra (\corref{enveloping corner}).

Throughout this section we assume that $\UU$ has finite quotients (condition~\eqref{finite})
and that $H$ acts effectively on $M$ (condition~\eqref{effective}).  
Recall that by definition,
\[
 \PP = \{ \Chi_{xU} : x\in N, U\in \UU \}
\quad\text{and}\quad
D = C^*(\PP) \subseteq \ell^\infty(N).
\]
Now we further define
\[
 D_0 = \spn\PP \subseteq D.
\]

\begin{lem}
$D_0$ is a $*$-subalgebra of $\ell^\infty(N)$, and consequently
\[
D
=\bar{D_0}
=\clspn\PP.
\]
\end{lem}

\begin{proof}

Clearly $D_0$ is a self-adjoint linear subspace of $\ell^\infty(N)$.
Fix $U,V\in\UU$, and 
choose $W\in \UU$ with $W\subseteq U\cap V$.
Then for any $x,y\in N$,
since $\UU$ has finite quotients
we can write
\[
 \Chi_{xU} = \sum_{zW \in xU/W} \Chi_{zW}
\quad\text{and}\quad
\Chi_{yV} = \sum_{wW\in yV/W}\Chi_{wW}
\]
where both sums are finite.
Thus 
\begin{align*}
\Chi_{xU}\Chi_{yV}
=\sum_{zW\in (xU\cap yV)/W}\Chi_{zW}
\in D_0,
\end{align*}
and it follows that $D_0$ is closed under multiplication.
\end{proof}

Now we want to work with the elements of $\PP$ more as projections in the $C^*$-algebra $D$ rather than as functions on $N$, so we introduce an alternative notation:

\begin{notn}
Put
\[
p_c=\Chi_{c}\for c\in N/U,U\in\UU.
\]
\end{notn}

\begin{lem}
\label{enveloping D}
$D_0$ has a unique $C^*$-norm, so that $D$ is the enveloping $C^*$-algebra of $D_0$.
\end{lem}

\begin{proof}
For any $*$-homomorphism $\pi$ of $D_0$ into a $C^*$-algebra $E$, the norm of $\pi(q)$ is at most $1$ for each $q\in\PP$, so $D_0$ has a universal enveloping $C^*$-algebra $C^*(D_0)$. For the uniqueness of the $C^*$-norm, we must show that if $\pi$ is a homomorphism of $C^*(D_0)$ into a $C^*$-algebra $E$, then $\pi$ is faithful if the restriction $\pi|_{D_0}$ is.
Equivalently, we must show that if $I$ is any nonzero ideal of $C^*(D_0)$, then $I\cap D_0\ne \{0\}$.

The key fact is that $C^*(D_0)$ is the closure of the union of the family of $C^*$-subalgebras
\[
D_U:=\clspn\{p_c:c\in N/U\}
\]
for $U\in\UU$ (where the closure is taken in $C^*(D_0)$),
which is an upward-directed family because $\UU$ is a filter base with finite quotients.
By a standard argument, it follows that there exists $U\in\UU$ such that $I\cap D_U\ne \{0\}$.
Now, the map
\[
f\mapsto \sum_{c\in N/U}f(c)p_c
\]
gives an isomorphism of $c_0(N/U)$ onto $D_U$.
It follows that
\[
I\cap \spn\{p_c:c\in N/U\}\ne \{0\}.
\]
Since
\[
\spn\{p_c:c\in N/U\}\subseteq D_0,
\]
we are done.
\end{proof}

Since $D_0$ is evidently self-adjoint and $G$-invariant,
the ``algebraic crossed product'' 
\[
D_0\rtimes_\alpha G := \spn i_G(G)i_D(D_0)
\]
is a $*$-subalgebra of the $C^*$-crossed product $D\rtimes_\alpha G$.
We emphasize that we use the term ``algebraic crossed product'' purely as shorthand
for $D_0\rtimes_\alpha G \subseteq D\rtimes_\alpha G$ as defined here.  There are certainly
other uses of the term in the literature, and they are generally different from ours;
our approach is inherently $C^*$-algebraic.

We will suppress the maps $i_G$ and $i_D$, thus identifying $G$ and $D_0$ with their images in $M(D\rtimes_\alpha G)$.
For $g\in G$, $U\in\UU$, and $c\in N/U$, covariance becomes
\[
 g p_c g\inv = p_{\beta_g(c)},
\]
where $\beta$ is as in~\eqref{U}.  

We first note that $D_0\rtimes_\alpha G$ is universal for covariant representations.
(A \emph{representation} of a $*$-algebra $B$ on a Hilbert space $X$ is a $*$-homomorphism from $B$ to $\BB(X)$.)

\begin{defn}
If $\pi$ and $u$ are representations of $D_0$ and $G$ on a Hilbert space $X$, we say $(\pi,u)$ is a \emph{covariant representation of $(D_0,G)$} if
\[
u_g\pi(f)u_g^*=\pi(\alpha_g(f))\righttext{for}g\in G,f\in D_0.
\]
\end{defn}

\begin{cor}\label{algebraic}
For every covariant representation $(\pi,u)$ of $(D_0,G)$ on a Hilbert space $X$, there is a unique representation $\Pi$ of $D_0\rtimes_\alpha G$ on $X$ such that
\begin{equation}
\label{integrated}
\Pi(gf)=u_g\pi(f) {\righttext{for}g\in G,f\in D_0}.
\end{equation}
\end{cor}

\begin{proof}
Uniqueness is clear, since $D_0\rtimes_\alpha G$ is spanned by the products $gf$ for $g\in G,f\in D_0$.
Given $(\pi,u)$, by \lemref{enveloping D} $\pi$ extends uniquely to a representation $\wilde\pi$ of $D$ on $X$. By density and continuity, the pair $(\wilde\pi,u)$ is a covariant representation of $(D,G)$, so there is a unique representation $\wilde\Pi$ of the $C^*$-crossed product $D\rtimes_\alpha G$ on $X$ such that
\[
\wilde\Pi(gf)=u_g\wilde\pi(f)\righttext{for}g\in G,f\in D.
\]
Then the restriction $\Pi:=\wilde\Pi|_{D_0\rtimes_\alpha G}$ is a representation of $D_0\rtimes_\alpha G$ on $X$ satisfying \eqref{integrated}.
\end{proof}

\begin{cor}
\label{enveloping crossed product}
$D\rtimes_\alpha G$ is the enveloping $C^*$-algebra of $D_0\rtimes_\alpha G$.
\end{cor}

\begin{proof}
This follows from
\lemref{enveloping D} and \cite[Lemma~2.3]{eq:full}.
\end{proof}

\begin{hyp}
\label{regular}
For the remainder of this section we assume that the action $\alpha$ of $G$ on $D$ is ``regular'' in the sense that the regular representation of $D\rtimes_\alpha G$ onto $D\rtimes_{\alpha,r} G$ is an isomorphism.
\end{hyp}

Note that Hypothesis~\ref{regular} is satisfied in particular whenever $G$ is amenable,
which is the case in all our examples.

As a consequence of \thmref{simple} and \hypref{regular}, the full crossed product $D\rtimes_\alpha G$ is simple and purely infinite.

\begin{cor}
The $*$-algebra $D_0\rtimes_\alpha G$ has a unique $C^*$-norm.
\end{cor}

\begin{proof}
This follows from \corref{enveloping crossed product} and simplicity of $D\rtimes_\alpha G$.
\end{proof}

\begin{notn}
For notational simplicity, let
\begin{itemize}
\item $A_0=D_0\rtimes_\alpha G$ (the algebraic crossed product);
\item $A=D\rtimes_\alpha G$ (the $C^*$-crossed product).
\end{itemize}
Also, let
\[
p=p_M.
\]
\end{notn}

The following lemma shows that $p$ is ``algebraically full'' in $A_0$:

\begin{lem}
$A_0=\spn A_0pA_0$.
\end{lem}

\begin{proof}
Since $gA_0=A_0$ for all $g\in G$, it suffices to show that for every $U\in\UU$ and $c\in N/U$ we have
\[
p_c\in \spn A_0pA_0.
\]
Choose $V\in\UU$ such that $V\subseteq U\cap M$. Then
\begin{align*}
p_c
&=\sum\{p_d:d\in N/V:d\subseteq c\}
\\&\in \spn\{q\in\PP:q\le p\}
\\&\subseteq \spn A_0pA_0,
\end{align*}
because $qp=q$ for all $q\in\PP$ with $q\le p$.
\end{proof}

Now we see that $A_0p$ is an $A_0-pA_0p$ imprimitivity bimodule in the sense of Fell and Doran \cite[Definition~XI.6.2]{fd2}.
Since $p\in A_0$, the left inner product ${}_L\<\cdot,\cdot\>$ on $A_0p$ is \emph{positive} in the sense that for all $b\in A_0$ we have
\[
{}_L\<bp,bp\>=(bp)(bp)^*.
\]
We need to know that the right inner product is also positive:

\begin{lem}
\label{positive}
For all $b\in A_0$ there exist $c_1,\dots,c_n\in pA_0p$ such that
\[
\<bp,bp\>_R=pb^*bp=\sum_{i=1}^nc_i^*c_i.
\]
\end{lem}

\begin{proof}
The proof is almost identical to an argument in \cite[Proof of Theorem~5.13]{hecke1}, so we will omit it.
\end{proof}

\begin{cor}
\label{enveloping corner}
$pA_0p$ has a unique $C^*$-norm, and
$pAp$ is its enveloping $C^*$-algebra.
\end{cor}

\begin{proof}
$A$ is simple by \thmref{simple} and \hypref{regular}, and is the enveloping $C^*$-algebra of $A_0$ by \corref{enveloping crossed product}.
Since $A$ is Morita-Rieffel equivalent to $pAp$ via the $A-pAp$ imprimitivity bimodule\footnote{in Rieffel's sense --- i.e., the inner products are positive} $Ap$, it follows that $pAp$ is simple.

Since the $A_0-pA_0p$ imprimitivity bimodule $A_0p$ is dense in the $A-pAp$ imprimitivity bimodule $Ap$, and the right inner product $\<\cdot,\cdot\>_R$ on $A_0p$ is positive by \lemref{positive},
an application of \cite[Proposition~5.5 (iii)]{hecke1} shows that $pAp$ is the enveloping $C^*$-algebra of $pA_0p$, and hence $pA_0p$ has a unique $C^*$-norm because $pAp$ is simple.
\end{proof}


\section{Embedding the Hecke algebra}\label{hecke}

We continue to assume that $\UU$ has finite quotients (condition~\eqref{finite})
and that $H$ acts effectively on $N$ (condition~\eqref{effective});
this implies that $M$ is  in fact a Hecke subgroup of $G$.
To see this, fix $g\in G$ and choose
$a\in H$ and $x\in N$ such that $g=ax$.  
Then $gMg\inv = aMa\inv \in \UU$ (since $M$ is normal in $N$), 
so by~\eqref{filter base} we can choose $U\in\UU$
such that $U\subseteq M\cap aMa\inv$,
and it follows from~\eqref{finite} that
\[
\bigl|M/(M\cap gMg\inv)\bigr| \leq \bigl| M/U \bigr| <\infty.
\]
Furthermore, by~\eqref{separating} we have
\[
 \bigcap_{g\in G} gMg\inv = \bigcap_{U\in\UU} U = \{e\};
\]
so the pair $(G,M)$ is \emph{reduced} in the sense of~\cite{hecke1},
and by~\cite[Lemma~6.3]{hecke1}, condition~\eqref{filter base}
means that the pair is \emph{directed}.

The \emph{Hecke algebra} $\MM$ of the pair $(G,M)$
is a convolution $*$-algebra generated by the double cosets of $M$ in $G$.
By \cite[Theorem~6.4]{hecke1}, 
$\MM$ has a universal enveloping $C^*$-algebra $C^*(\MM)$;
this is the \emph{Hecke $C^*$-algebra} 
of the pair $(G,M)$.

In this section,
we will give sufficient conditions for $C^*(\MM)$
to embed in the Cuntz-Li algebra.
In Section~\ref{BC}, this will be applied to 
the generalized Bost-Connes algebra.

Our embedding will require $\bar N$ to be self-dual:

\begin{thm}
\label{nonabelian}
Let $N$ be abelian, and assume that there exists an isomorphism $\theta:\bar N\to \what{\bar N}$ such that
\begin{align}
\label{sufficient}
&\theta\circ \beta_a(n)=\theta(n)\circ \beta_a\inv
\midtext{for}a\in H,n\in \bar N,\midtext{and}
\\
\label{projection}
&\theta\bigl(\bar M\bigr)=\bigl(\bar M\bigr)^\perp.
\end{align}
Then
the Hecke $C^*$-algebra $C^*(\MM)$ embeds faithfully in the 
corner
$p(D\rtimes_\alpha G)p$, where $p=i_D(\Chi_M)$.
\end{thm}

\begin{proof}
The subgroup topology on $G$ determined by $\UU$ 
is precisely the \emph{Hecke topology} of the pair $(G,M)$
(\cite[Definition~3.3]{hecke1}),
so the Hausdorff completion $\bar G$ of $G$
with respect to this topology is a locally compact
(and totally disconnected) group,
the closure $\bar M$ of $M$ in $\bar G$ is 
a compact open subgroup of $\bar G$,
and the Hausdorff completion $\bar N$ of $N$
is (identified with) the closure of $N$ in $\bar G$
(\cite[Section~3]{hecke1}).  
We normalize the left Haar measure on $\bar G$
so that $\bar M$ has measure~$1$;
thus $\Chi_{\bar M}$ is a projection 
in the convolution $*$-algebra $C_c(\bar G)$,
and $\MM$ can be identified with the corner $\Chi_{\bar M}C_c(\bar G)\Chi_{\bar M}$.
Moreover, since $(G,M)$ is directed,
$\Chi_{\bar M}$ is a full projection in the group $C^*$-algebra $C^*(\bar G)$,
and we can identify $C^*(\MM)$
with the full corner $\Chi_{\bar M} C^*(\bar G)\Chi_{\bar M}$ (\cite[Theorem~6.4]{hecke1}).

Now, the $ax+b$ group action $\beta$ of the discrete semidirect product $G=N\rtimes H$ by homeomorphisms of the space $\bar N$ (extended from the $ax+b$ group action on $N$, where we now drop the bar on the notation for the extended action) restricts to the action of $H$ by automorphisms of $\bar N$ that defines the semidirect product $\bar G=\bar N\rtimes H$, and we continue to denote this action by $\beta:H\to \aut \bar N$.
This in turn determines an action $\gamma:H\to \aut C^*(\bar N)$ such that
\[
C^*(\bar G)\cong C^*(\bar N)\rtimes_\gamma H,
\]
and the isomorphism carries $\Chi_{\bar M}$ to $i_{C^*(\bar N)}(\Chi_{\bar M})$.

On the other side,
we will replace $D$ with the isomorphic $C^*$-algebra $C_0(\bar N)$ (thus replacing $i_D(\Chi_M)$ with $i_{C_0(\bar N)}(\chim)$), and we will denote the associated action of $G$ on $C_0(\bar N)$ by $\alpha$ (rather than $\bar\alpha$ as we did previously).
Thus, to prove the theorem it suffices to find an isomorphism
\[
C^*(\bar N)\rtimes_\gamma H\cong C_0(\bar N)\rtimes_\alpha H
\]
that takes $i_{C^*(\bar N)}(\chim)$ to $i_{C_0(\bar N)}(\chim)$,
and for this it suffices to find a $\gamma-\alpha$ equivariant isomorphism
{that takes $C^*(\bar N)$ onto $C_0(\bar N)$ and preserves $\chim$.}

We claim that the isomorphism $\rho$ defined by the commutative diagram
\[
\xymatrix{
C^*(\bar N) \ar[r]^\rho \ar[d]_{f\mapsto \what f}
&C_0(\bar N)
\\
C_0(\what{\bar N}) \ar[ur]_{f\mapsto f\circ \theta}
}
\]
does the job,
where $\what f$ denotes the Fourier transform of $f$, for which we use the convention
{(again normalizing so that $\int_{\bar N}\Chi_{\bar M}(n)\,dn = 1$) that}
\[
\what f(\chi)=\int_{\bar N} f(n)\bar{\chi(n)}\,dn.
\]

In preparation for the verification of this claim, we record the formula for $\gamma$:
for $a\in H$ and $f\in C_c(\bar N)$ we have
\begin{align*}
\gamma_a(f)
&=\int_{\bar N} f(n)\beta_a(n)\,dn
\\&=\Delta_\beta(a)\int_{\bar N} f(\beta_a\inv(n))n\,dn
&&\text{(for some scalar $\Delta_\beta(a)$)}
\\&=\Delta_\beta(a)f\circ \beta_a\inv.
\end{align*}
Then for the same $a,f$, and for $n\in \bar N$, we have
\begin{align*}
\rho\bigl(\gamma_a(f)\bigr)(n)
&=\bigl(\gamma_a(f)\spwhat\circ \theta\bigr)(n)
\\&=\Delta_\beta(a)(f\circ \beta_a\inv)\spwhat(\theta(n))
\\&=\Delta_\beta(a)\Delta_\beta(a)\inv\what f\bigl(\theta(n)\circ \beta_a\bigr)
\\&=\what f\bigl((\theta\circ \beta_a\inv)(n)\bigr)
&&\text{(by \eqref{sufficient})}
\\&=\bigl(\what f\circ\theta\bigr)\bigl(\beta_a\inv(n)\bigr)
\\&=\rho(f)\bigl(\beta_a\inv(n)\bigr)
\\&=\alpha_a\bigl(\rho(f)\bigr)(n),
\end{align*}
where the third equality follows from the following calculation: for $\chi\in \what{\bar N}$ we have
\begin{align*}
(f\circ \beta_a\inv)\spwhat(\chi)
&=\int_{\bar N}f\circ \beta_a\inv(n)\bar{\chi(n)}\,dn
\\&=\Delta_\beta(a)\inv \int_{\bar N}f(n)\bar{\chi(\beta_a(n))}\,dn
\\&=\Delta_\beta(a)\inv\what f(\chi\circ \beta_a).
\end{align*}

Thus $\rho$ is a $\gamma-\alpha$ equivariant isomorphism.
Since the Fourier transform of $\chim$ is $\Chi_{\bar M^\perp}$,
and since our hypothesis 
\eqref{projection}
implies that
$\Chi_{\bar M^\perp}\circ\theta=\chim$,
we have
\[
\rho(\chim)=\chim,
\]
as required.
\end{proof}

\subsection*{The abelian case}

We now assume that $H$ is abelian, and derive an alternative sufficient condition for embedding:

\begin{thm}
\label{abelian}
Let both $N$ and $H$ be abelian, and assume that there exists an isomorphism $\theta:\bar N\to \what{\bar N}$ such that
\begin{align}
\label{alternative}
&\theta\circ \beta_a(n)=\theta(n)\circ \beta_a
\midtext{for}a\in H,n\in \bar N,\midtext{and}
\\
\label{projection again}
&\theta\bigl(\bar M\bigr)=\bigl(\bar M\bigr)^\perp.
\end{align}
Then
the Hecke $C^*$-algebra $C^*(\MM)$ embeds faithfully in the 
corner
$p(D\rtimes_\alpha G)p$, where $p=i_D(\Chi_M)$.
\end{thm}

\begin{proof}
Recall from the proof of \thmref{nonabelian} that we have an action $\beta:H\to \aut \bar N$ giving the semidirect product $\bar G=\bar N\rtimes_\beta H$, and that it suffices to find an embedding of $C^*(\bar G)$ in $C_0(\bar N)\rtimes_\alpha H$ taking $\chim$ to $i_{C_0(\bar N)}(\chim)$.
Since $H$ is abelian, we can define another action $\beta':H\to \aut \bar N$ by
\[
\beta'_a=\beta_a\inv.
\]
A routine calculation shows that the assignment
\[
(n,a)\mapsto (n,a\inv)
\]
gives an isomorphism of semidirect products:
\[
\bar N\rtimes_\beta H\cong \bar N\rtimes_{\beta'} H,
\]
so it suffices to embed
\[
C^*(\bar N\rtimes_{\beta'} H)\cong C^*(\bar N)\rtimes_{\gamma'} H
\]
in $C_0(\bar N)\rtimes_\alpha H$,
where $\gamma':H\to \aut C^*(N)$ is the action associated to $\beta'$.
Again, it suffices to show that the same isomorphism $\rho:C^*(N)\cong C_0(N)$ as we used in the proof of \thmref{nonabelian} is now $\gamma'-\alpha$ equivariant and preserves $\chim$.
But the same calculations as in that proof accomplishes this, using the modified hypothesis \eqref{alternative} rather than \eqref{sufficient}.
\end{proof}

\begin{rem}
\label{symmetric}
Conditions \eqref{sufficient} -- \eqref{projection again} can be expressed using a bicharacter: 
assuming that $N$ is abelian, and that there is an isomorphism $\theta:\bar N\to \what{\bar N}$,
we can define a bicharacter on $\bar N$ by
\[
B(x,y)=\theta(y)(x)\midtext{for}x,y\in \bar N.
\]
Then \eqref{sufficient} is equivalent to
\begin{equation}
\label{bicharacter1}
B(\beta_a(x),\beta_a(y))=B(x,y)\midtext{for}a\in H,x,y\in \bar N.
\end{equation}
Condition~\eqref{alternative} is equivalent to symmetry of the automorphisms $\beta_a$:
\begin{equation}
\label{bicharacter}
B(\beta_a(x),y)=B(x,\beta_a(y))\midtext{for}a\in H,x,y\in \bar N,
\end{equation}
and \eqref{projection} and \eqref{projection again} are equivalent to:
\begin{equation}
\label{M selfdual}
	x\in \bar M \iff B(x,y)=1 \midtext{for all}y\in \bar M.
\end{equation}
\end{rem}

\section{Integral domains}
\label{integral}

In 
Sections~\ref{prelim} and \ref{G} 
we introduced a general context where $G=N\rtimes H$ is a semidirect product with certain properties, and $\alpha$ is an action of $G$ on a commutative $C^*$-algebra $D$. A source of examples comes from integral domains, as in the work of Cuntz and Li \cite{CLintegral}.
Let $R$ be an integral domain that is not a field, 
so that in particular its group of units is strictly contained in the set $R\minus\{0\}$ of nonzero elements,
and assume that $R$ has \emph{finite quotients} in the sense that
\begin{equation}\label{ring hyp}
\bigl|R/aR\bigr|<\infty\all a\in R\minus\{0\}.
\end{equation}

Let $Q=Q(R)$ be the field of fractions of $R$,
and in the notation of the previous sections,
take:
\begin{itemize}
\item $N=Q$ (the additive group of $Q$);

\item $H=Q^\times$ (the multiplicative group of the field $Q$);

\item $M=R$ (the additive group of the ring $R$).
\end{itemize}
Thus $G= N\rtimes H = Q\rtimes Q^\times$ is the $ax+b$ group of $Q$, 
and our assumptions 
\eqref{separating}--\eqref{effective}
hold in this situation.
Moreover, because $G$ is amenable in this context we have $D\rtimes_{\alpha,r} G=D\rtimes_\alpha G$.

\begin{thm}\label{domain theorem}
With the above notation, the crossed product $D\rtimes_\alpha G$ is simple and purely infinite.
\end{thm}

\begin{proof}
This follows immediately from 
\thmref{simple} 
and amenability of $G$.
\end{proof}

In \secref{ring} below we will show that the Cuntz-Li algebra $\fA[R]$ of \cite{CLintegral} is isomorphic to a corner of $D\rtimes_\alpha G$.

\section{Embedding the generalized Bost-Connes algebra}
\label{BC}

Let $K=\Q(\theta)$ be an algebraic number field with $\theta$ an algebraic integer, 
let $N$ and $M$ be the additive groups of $\Q(\theta)$ and $\Z[\theta]$,
respectively,
and let
$H$ be the multiplicative group of $K$.
In this situation the Hecke $C^*$-algebra $C^*(\MM)$, as defined in \secref{hecke},
is the generalized Bost-Connes $C^*$-algebra.  
In combination with \thmref{Cuntz Li}, 
Corollary~\ref{embed} below will recover
Cuntz and Li's embedding of the Bost-Connes algebra in $\fA[R]$.

Sometimes,
$M$ equals the ring $R$ of integers in $K$; 
if so, $K$ is called a {\it monogenic} field.
In general there is $s\in\N$ such that $s R\subset M\subset R$. 
Therefore the topologies defined by  $M$ and $R$ will be the same,
in particular the completion $\bar N$ is  the ring $\AA_f$ of finite adeles of $K$, \cf\  \cite[Chapter VII, \S2, no. 4, Prop. 3]{bourbaki:CA} and the preceding discussion there.

\begin{cor}
\label{embed}
If  $M=R$ and with the above notation,  the generalized Bost-Connes algebra $C^*(\MM)$ embeds faithfully in the corner $p(D\rtimes_\alpha G)p$, where $p=i_D(\Chi_M)$.
\end{cor}

For this we need the following:

\begin{lem}
\label{selfdual}
There is a $\Q$-linear map $\phi:K\to \Q$ such that 
	\[x\in M \iff \phi(xy)\in \Z \text{ for all }y\in M.
\]
\end{lem}
\begin{proof} Let $p(x)=\alpha_0+   \alpha_1 x+ \cdots \alpha_{n-1} x^{n-1}-x^n$ 
with all $\alpha_i\in\Z$ be $\theta$'s minimal polynomial. Define 
$\phi:K\to \Q$ by 
	\[
	\phi(a_0+   a_1 \theta+ \cdots a_{n-1} \theta^{n-1})=a_{n-1}.
\]
Clearly 
\begin{align*}
\phi(\theta^i)=	&
\begin{cases}
0 &\text{ for } 0\leq i <n-1\\
1 &\text{ for }  i =n-1\\
s_i\in\Z &\text{ for } n\leq i,
\end{cases}
\end{align*}
so $\phi(xy)\in \Z $ for all $x,y\in M$.

Now suppose $x=a_0+   a_1 \theta+ \cdots a_{n-1} \theta^{n-1}\in K$ satisfies
$\phi(xy)\in \Z $ for all $y\in M$, in particular 
$\phi(x\theta^i)\in \Z $ for all $i$. Checking this for $i=0,1,2,\cdots$ we get
\begin{align*}
a_{n-1}&\in \Z\\
a_{n-2}&+a_{n-1}\phi(\theta^n)\in \Z\\
a_{n-3}&+a_{n-2}\phi(\theta^n)+a_{n-1}\phi(\theta^{n+1})\in \Z	
\end{align*}
\etc \
Since $\phi(\theta^i)\in\Z$ for all $i$, we get $a_{n-1}\in \Z$, then
$a_{n-2}\in \Z$, \etc \ So $x\in M $.
\end{proof}

\begin{proof}[Proof of Corollary~\ref{embed}]
By \thmref{abelian},
we only need to verify 
Hypotheses~\eqref{bicharacter} and \eqref{M selfdual}.
With $\phi$ as in Lemma~\ref{selfdual}, define a character $\lambda$ on $K$ by
	\[
\lambda(x)=\exp(2\pi i\phi(x)),
\]
and a bicharacter $B$ on $K$ by
	\[B(x,y)=\lambda(xy).
\]
Copying the proof of
 \cite[Chapter~XIV, Section~1, Theorem~1]{LangNumberTheory}, 
 $\lambda$  extends to a character of  $\bar N$, and $B$ extends to a bicharacter inducing a self-duality of  $\bar N$.

Clearly $B(x,y)=1$ for all $x,y\in \bar M$.
Conversely, given $x\in \bar N$, we have $x=x_0+m$ with  $x_0\in N$ and $m\in \bar M$. Therefore
\begin{align*}
	& B(x,y)=1 \midtext{for all}y\in \bar M\\
\implies & \lambda((x_0+m)y)=1 \midtext{for all}y\in  M\\
\implies & \lambda(x_0y)=1 \midtext{for all}y\in  M\\
\implies & \phi(x_0y)\in\Z \midtext{for all}y\in  M\\
\implies & x_0\in  M\\
\implies & x\in  \bar M.
\end{align*}
Thus \eqref{M selfdual} holds.
Since $H=K^\times$ acts on $\bar N$  by multiplication (using the canonical embedding of $K$ in $\bar N$), \eqref{bicharacter}
  is  also satisfied.
\end{proof}


\section{lamplighter groups}
\label{further}

Suppose $H$ is an infinite abelian group with $H^+$ an Ore subsemigroup, \ie
	\[ 
	(H^+)\inv \cap (H^+)=\{e\}\quad \text{and}\quad (H^+)\inv H^+=H.
\]
We also assume the following finiteness condition:
	\[
	c\in H^+ \implies H^+ \setminus c H^+ \quad\text{is finite.}
\]

Let $F$ be a finite group and take
	\[N=\bigoplus_H F
	=\{f:H\to F\,|\, \supp(f) \text{ is finite } \}
\]
with pointwise multiplication, where $\supp(f)=\{x\in H| f(x)\neq e\}$.

Then take
	\[M=\bigoplus_{H^+} F
	=\{f\in N\,|\, \supp(f) \subseteq H^+ \},
\]
clearly a normal subgroup of $N$. $H$ acts on $N$ by shifting:
	\[
	_af(x)=f(a\inv x).
\]
Then $G=N\rtimes H$  is called the {\it wreath product} of $H$ and $F$ (\cf\ \cite[pp.~172--176]{Rotman}),
or the {\it lamplighter group} (\cf\ \cite{harpe-topics})
if $H=\Z$ with $H^+=\N$. 

One checks that 
	\[aMa\inv
	=\{f\in N\,|\, \supp(f) \in aH^+ \},
\]
and that $a\in H^+\iff aMa\inv\subseteq M$, so the notation is consistent with \secref{gnr}.

If $a=b\inv c$ with $b,c\in H^+$, then $aMa\inv\cap M\supseteq cMc\inv$, so 
$\{aMa\inv\}$ is downward directed.

Furthermore, if $c\in H^+$ then
	\[M/cMc\inv
	\cong\{f\in N\,|\, \supp(f) \in H^+ \setminus c H^+ \},
\]
which is finite.

If $f\in\cap \,aMa\inv$, then $\supp(f) \subseteq\cap \,aH^+=\emptyset$ (the last equality is not obvious), so $\cap\, aMa\inv=\{e\}$.

$H$ acts effectively on $M$, so all assumptions in \secref{prelim} are satisfied. In addition, it should be clear that $\cup\, aMa\inv=N$.

As to the completions, we see that 
	\[\overline M=\{f:H^+\to F \},
\]
and 
	\[\overline N=\bigcup_{a\in H^+}\{f:a\inv H^+\to F \}.
\]
So with $F$ non-abelian, this gives examples with $\overline N$ non-abelian.

If $F$ is abelian, then $F\cong\widehat F$ by some isomorphism $\theta$. Take
	\[B(f,g)=\prod_i \langle f(i),\theta(g(i\inv))\rangle
\]
and note the nontrivial fact that for $f,g\in\overline N$ the product is finite.
This bilinear form is symmetric, \ie\ satisfies \eqref{bicharacter}. 
Moreover, 
	\[f\in \overline M\iff B(f,g)=1 \text{ for all } g \in \overline M.
\]
Thus \thmref{abelian} can be applied in this situation.


\section{Generators and relations}
\label{gnr}

In this section we shall look at the 
crossed products
$A=D\rtimes G$ and $C=D\rtimes H$ together with the corner subalgebras $pAp$ and $pCp$,
where $p=\Chi_M$.
We shall see 
that,
under \hypref{abelian hyp} below,
the corner algebras have generators satisfying relations \`a la \cite{cuntzq, CLintegral} and that they in fact are universal for these generators and relations.
It turns out that the result for $pAp$ follows from the case of $pCp$, so we will deal with $pCp$ first.
There is some commonality in our approach to the two cases:
we start with a representation of the generators and relations on a Hilbert space $X$. We use a dilation technique to embed  $X$ in a larger Hilbert space $\wilde X$ where we can represent the generators of the full algebras $C$ respectively $A$.  $C$ and $A$ are universal for covariant representations, so finally we only have to cut down with the projection $p$ to get the result.

We continue to assume condition~\eqref{finite}, that $\UU$ has finite quotients;
condition~\eqref{effective}, that $H$ acts effectively on $N$, is not needed for this section.

To begin, let
\[
H^+=\{a\in H:aMa\inv\subset M\}.
\]
From our assumptions it follows that given $h\in H$ there exists $a\in H$ such that
\[
aMa\inv\subset hMh\inv\cap M,
\]
so $a\in H^+$ and $h\inv a \in H^+$, and thus
\[
h=a(h\inv a)\inv\in H^+(H^+)\inv.
\]
It follows that $H$ is directed by the relation
\[
a\le b\midtext{if and only if}b\in aH^+.
\]
Observe that this relation is not a partial ordering in general, since
$ H^+\cap (H^+)\inv$ can be nontrivial.
\begin{obs}
For $a\in H$ we have $a\in H^+$ if and only if $ap=pap$,
where we remind the reader that we identify elements of the
groups $H$ and $G$ with their images in the multipliers of the crossed
products $D\rtimes H$ and $D\rtimes G$.
\end{obs}

\begin{hyp}
\label{abelian hyp}
Throughout this section we assume that
$H$ is abelian and $N=\bigcup_{a\in H^+} a\inv Ma$.
\end{hyp}
This hypothesis will be satisfied in all of our examples.

We adopt the standard conventions that in a unital $C^*$-algebra $B$, an \emph{isometry} is an element $s$ such that $s^*s=1$, a \emph{unitary} is an element $u$ such that $u^*u=uu^*=1$, and a \emph{projection} is an element $p$ such that $p=p^*=p^2$. 

The following relations will be useful for constructing representations of $pCp$:

\begin{defn}
An \emph{SP-family} in a unital $C^*$-algebra $B$ consists of a set $\{S_a:a\in H^+\}$ of isometries in $B$ and a set $\{P(a,m):a\in H^+,m\in M\}$ of projections satisfying the relations\begin{align}
\label{s rep}
&S_aS_b=S_{ab}
\\
\label{p e}
&P(e,e)=1
\\
\label{s p}
&S_a P(b,m)S_a^*=P(ab,ama\inv)
\\
\label{p consistent}
&P(a,k)=\sum_{mbMb\inv\in M/bMb\inv} P(ab,kama\inv){\midtext{for all}b\in H^+}.
\end{align}
\end{defn}

We shall also need the following for constructing representations of~$D$:

\begin{defn}
A \emph{PN-family} in a $C^*$-algebra $B$ consists of a set $\{P(a,n):a\in H^+,n\in N\}$ of projections satisfying the relations
\begin{align}
\label{orthogonal}
&P(a,n)P(a,k)=0\midtext{if}naMa\inv\ne kaMa\inv;
\\
\label{pn consistent}
&P(a,n)=\sum_{mbMb\inv\in M/bMb\inv}P(ab,nama\inv)\midtext{for all}b\in H^+.
\end{align}
\end{defn}

It will sometimes be convenient to have analogues of the above relations with $H^+$ replaced by $H$:

\begin{defn}
A \emph{PNH-family} in a $C^*$-algebra $B$ consists of
a set $\{P(h,n):h\in H,n\in N\}$ of projections satisfying the relations
\begin{align}
\label{h orthogonal}
&P(h,n)P(h,k)=0\midtext{if}nhMh\inv\ne khMh\inv;
\\
\label{pnh consistent}
&P(h,n)=\sum_{mbMb\inv\in M/bMb\inv}p(hb,nhmh\inv)\midtext{for all}b\in H^+.
\end{align}
\end{defn}

The following relations are needed to get representations of $pAp$:

\begin{defn}
An \emph{SU-family} in a unital $C^*$-algebra $B$ consists of a set $\{S_a:a\in H^+\}$ of isometries in $B$ and a set $\{U(m):m\in M\}$ of unitaries in $B$ satisfying the relations~\eqref{s rep} and
\begin{align}
\label{s u}
&S_aU(m)=U(ama\inv)S_a
\\
\label{u consistent}
&\sum_{maMa\inv\in M/aMa\inv} U(m)S_aS_a^*U(m\inv)=1
\\
\label{u rep}
&U(m)U(k)=U(mk).
\end{align}
\end{defn}

Thus, $S$ is an isometric representation of $H^+$ in $B$,
and also $U$ is a unitary representation of $M$ in $B$.

Note that the $C^*$-algebra $pCp$ is unital (with unit $p=\Chi_M$);
we construct an SP-family in $pCp$:

\begin{lem}\label{SP in pCp}
For $a\in H^+$ and $m\in M$ define elements of $pCp$ by
\begin{itemize}
\item $s_a=pap$;

\item $p(a,m)=\Chi_{maMa\inv}$.
\end{itemize}
Then $\{s_a,p(a,m):a\in H^+,m\in M\}$ is an SP-family that generates $pCp$
{as a $C^*$-algebra}.
\end{lem}

\begin{proof}
The verification that $\{s_a,p(a,m)\}$ is an SP-family is routine, with the help of the elementary properties
\begin{itemize}
\item[{}] $p(a,m)=pmapa\inv m\inv p=pms_as_a^*m\inv p$.
\end{itemize}

To see that this SP-family generates $pCp$, note that
\[
C=\clspn\{a\inv b \Chi_{ncMc\inv}:a,b,c\in H^+,n\in N\},
\]
so
\[
pCp
=\clspn\{pa\inv b \Chi_{ncMc\inv}p:a,b,c\in H^+,n\in N\}.
\]
Now,
\[
\Chi_{ncMc\inv}p=\begin{cases}
p(c,n)\case n\in M\\
0\case n\notin M.
\end{cases}
\]
Thus
\[
pCp
=\clspn\{s_a^*s_b p(c,m):a,b,c\in H^+,m\in M\}.
\qedhere
\]
\end{proof}

We similarly construct an SU-family in the unital $C^*$-algebra $pAp$:

\begin{lem}\label{SU in pAp}
For $a\in H^+$ and $m\in M$ define elements of $pAp$ by
\begin{itemize}
\item $s_a=pap$;

\item $u(m)=pmp$.
\end{itemize}
Then $\{s_a,u(m):a\in H^+,m\in M\}$ is an SU-family that generates $pAp$ as a $C^*$-algebra.
\end{lem}

\begin{proof}
The verification that $\{s_a,u(m)\}$ is an SU-family is routine, using the elementary properties
\begin{itemize}
\item $pap=ap$;
\item $s_as_a^*=apa\inv=\Chi_{aMa\inv}$;
\item $u(m)=mp=pm$.
\end{itemize}

To see that this SU-family generates $pAp$, note that
\begin{multline*}
A
=\clspn\{a\inv br\inv mr\Chi_F:
\\a,b,r,c\in H^+,m\in M,F=ncMc\inv, n\in N\},
\end{multline*}
so
\begin{multline*}
pAp
=\clspn\{pa\inv br\inv mrp\Chi_F:
\\a,b,r,c\in H^+,F=m_1cMc\inv ,m,m_1\in M \}.
\end{multline*}
The result now follows from 
\begin{align*}
pa\inv br\inv mrp
&=s_{ra}^*s_bu(m)s_r
\\\Chi_{m_1cMc\inv}	&=u(m_1)s_cs_c^*u(m_1\inv).
\qedhere
\end{align*}
\end{proof}

\begin{rem*}
Our notation for the above SP-family and SU-family is similar to that in \cite{CLintegral}.
\end{rem*}

We also construct a PN-family in $D$:

\begin{lem}\label{PN in D}
For $a\in H^+$ and $n\in N$ define an element of $D$ by
\begin{itemize}
\item $p(a,n)=\Chi_{naMa\inv}$.
\end{itemize}
Then $\{p(a,n):a\in H^+,n\in N\}$ is a PN-family that generates $D$ as a $C^*$-algebra.
\end{lem}

\begin{proof}
This is routine.
\end{proof}

The following two theorems are the main results of this section: The above relations give generators and relations for $pCp$ and $pAp$, in a sense made precise below.

\begin{thm}\label{pCp}
Assuming \hypref{abelian hyp}, 
for every SP-family $\{S_a,P(a,m):a\in H^+,m\in M\}$ in a unital $C^*$-algebra $B$
there is a unique homomorphism $\Pi:pCp\to B$ such that
\begin{align*}
\Pi(s_a)&=S_a&&\text{for all }a\in H^+\\
\Pi(p(a,m))&=P(a,m)&&\text{for all }a\in H^+,m\in M.
\end{align*}
\end{thm}

In other words, \thmref{pCp} says that $pCp$ is a universal $C^*$-algebra for relations  \eqref{s rep}--\eqref{p consistent}.

The proof will go as follows:
we can put the $C^*$-algebra $B$ on Hilbert space, so that we have
isometries $S_a$ and projections $P(a,m)$ on a Hilbert space $X$
satisfying \eqref{s rep}--\eqref{p consistent}.
We shall construct a Hilbert space $\wilde X$ together with 
unitary operators 
$\wilde S_h$ for $h\in H$ and  
projections $\wilde P(h,n)$ for $h\in H$ and $n\in N$ such that
$\lambda: \Chi_{nhMh\inv}\mapsto \wilde P(h,n)$ is a  representation of 
$D_0$ (defined in \secref{universal-sec})
with $(\lambda, \wilde S)$ covariant.

In fact, we use Douglas' construction \cite{RGD} of the Hilbert space $\wilde X$, which only depends upon the isometries $S_a$, and Douglas' dilation theorem \cite[Theorem~1]{RGD} shows {that we} get a homomorphism $\wilde S$ from $H^+$ into the unitary group of $\wilde X$ and an isometric embedding $T:X\to \wilde X$ intertwining $S$ and $\wilde S$:
\[
\wilde S_aT=TS_a\midtext{for}a\in H^+;
\]
we will extend $\wilde S$ in the obvious way to a unitary representation of $H$.
The embedding $T$ will also intertwine $P$ and $\wilde P$:
\[
\wilde P(a,m)T=TP(a,m)\midtext{for}a\in H^+,m\in M.
\]
$(\lambda, \wilde S)$ extends to a representation $\wilde\Pi$ of $C$ and we get the desired representation $\Pi$ of $pCp$  by
$\Pi(z)=T^*\wilde\Pi(z)T$.

\begin{thm}\label{pAp}
Assuming \hypref{abelian hyp}, 
for every SU-family $\{S_a,U(m):a\in H^+,m\in M\}$ in a unital $C^*$-algebra $B$
there is a unique homomorphism $\Pi:pAp\to B$ such that
\begin{align*}
\Pi(s_a)&=S_a&&\text{for all $a\in H^+$}\\
\Pi(u(m))&=U(m)&&\text{for all $m\in M$.}
\end{align*}
\end{thm}

In other words, \thmref{pAp} says that $pAp$ is a universal $C^*$-algebra for relations~\eqref{s rep}
and \eqref{s u}--\eqref{u rep}.

For this theorem we first appeal to \thmref{pCp},
after applying \lemref{SU to SP} to get an SP-family,
to get $\wilde S$ and $\wilde P$ as before, and then
we will construct unitary operators 
$\wilde U(n)$ for $n\in N$ such that $(h,n)\mapsto \wilde S_h\wilde U(n)$ is a unitary representation of $G$.

Here the same isometry $T: X\mapsto \wilde X$ also intertwines $U$ and $\wilde U$:
\[
\wilde U(m)T=U(m)T\midtext{for}m\in M.
\]
$\wilde S\wilde U$ will determine a representation $\wilde\Pi$ of $A$, and again we get the desired representation of $pAp$  by
$\Pi(z)=T^*\wilde\Pi(z)T$.

\subsection*{Consequences of the relations}

Our proofs of Theorems~\ref{pCp}--\ref{pAp} will use techniques that fall into two types: consequences of the relations, and dilation techniques.
Here we separate out the techniques that do not involve dilation.

\begin{lem}\label{SU to SP}
Every SU-family $\{S_a,U(m):a\in H^+,m\in M\}$ gives rise to an SP-family $\{S_a,P(a,m):a\in H^+,m\in M\}$ via
\[
P(a,m)=U(m)S_aS_a^*U(m)^*.
\]
\end{lem}

\begin{proof}
This follows quickly from relations \eqref{s rep}--\eqref{p consistent} and \eqref{s u}--\eqref{u rep}.
\end{proof}

\begin{lem}\label{H^+ to H}
Let $B$ be a unital $C^*$-algebra.
Suppose we have unitaries $\{S_a:a\in H^+\}$ in $B$ satisfying relation~\eqref{s rep}.
Then $S$ extends uniquely to a unitary representation of $H$ in $B$.
\end{lem}

\begin{proof}
If $h=ab\inv$ with $a,b\in H^+$, define
\[
S_h=S_aS_b^*.
\]
If $h=ab\inv=cd\inv$ with $a,b,c,d\in H^+$, then $ad=da=cb$, so
\begin{align*}
S_aS_b^*
&=S_aS_dS_d^*S_b^*
=S_{ad}S_{bd}^*
=S_{cb}S_{bd}^*
=S_cS_bS_b^*S_d^*
=S_cS_d^*,
\end{align*}
thus $S_h$ is well-defined. It is of course unitary, and so,
because $S_a$ and $S_b$ are commuting unitaries,
$S_a$ and $S_b^*$ commute for $a,b\in H^+$, whence $S$ is a homomorphism.
\end{proof}

\begin{lem}\label{SP to PN}
Let $B$ be a unital $C^*$-algebra.
Suppose we have unitaries $\{S_a:a\in H^+\}$ and projections $\{P(a,m):a\in H^+,m\in M\}$ in $B$ satisfying relations \eqref{s rep}, \eqref{s p}, and \eqref{p consistent}.
Then $P$ extends uniquely to a PN-family $\{P(a,n):a\in H^+,n\in N\}$ such that
\begin{equation}\label{define PN}
P(a,c\inv mc)=S_c^*P(ca,m)S_c\midtext{for}c\in H^+.
\end{equation}
Moreover, this extended family satisfies the analogue of \eqref{s p}:
\begin{equation}\label{s pn}
S_aP(b,n)S_a^*=P(ab,ana\inv)\midtext{for}a,b\in H^+,n\in N.
\end{equation}
\end{lem}

\begin{proof}
Since $N$ is the union of the upward-directed family $\{c\inv Mc:c\in H^+\}$ (with $c\inv Mc\subset d\inv c\inv Mdc$), to construct a map from $H^+\times N$ to the projections in $B$ it suffices to find, for each $c\in H^+$, a map
$P^c$ from $H^+\times c\inv Mc$ to the projections, such that
\[
P^{dc}|_{H^+\times c\inv Mc}=P^c.
\]
For $a\in H^+$ and $m\in M$ we define $P^c(a,c\inv mc)$ by the right-hand side of \eqref{define PN}.
This is well-defined, because $m\mapsto c\inv mc$ is a bijection of $M$ onto $c\inv Mc$,
and we have
\begin{align*}
P^c(a,c\inv mc)
&=S_c^*P(ca,m)S_c
\\&=S_{dc}^*S_dP(ca,m)S_c^*S_{dc}
\\&=S_{dc}^*P(dca,dmd\inv)S_{dc}
\\&=P^{dc}(a,c\inv mc).
\end{align*}
In particular, $P^c|_{H^+\times M}=P$.
Thus we have defined a family $\{P(a,n):a\in H^+,n\in N\}$ of projections extending the given family $\{P(a,m):a\in H^+,m\in M\}$ and satisfying \eqref{define PN}, and moreover this relation gives the uniqueness.
	
We parlay the relation \eqref{p consistent} into the relation \eqref{pn consistent} for 
these projections:
if $n=c\inv kc$ for $c\in H^+$ and $k\in M$ then
\begin{align*}
P(a,c\inv kc)
&=S_c^*P(ca,k)S_c
\\&=S_c^*\sum_{mbMb\inv\in M/bMb\inv}P(cab,kcama\inv c\inv)S_c
\\&=\sum_{mbMb\inv\in M/bMb\inv}P(ab,c\inv kcama\inv).
\end{align*}

For \eqref{orthogonal}, fix $a\in H^+$, and let $n,k\in N$ with $naMa\inv\ne kaMa\inv$.
We must show that the projections $P(a,n)$ and $P(a,k)$ are orthogonal.
Choose $c$ in $H^+$ such that $cnc\inv,ckc\inv\in M$.
It suffices to show that
\[
S_cP(a,n)S_c^*=P(ca,cnc\inv)
\quad\perp\quad
S_cP(a,k)S_c^*=P(ca,ckc\inv).
\]
Put $b=ca$. Since
\[
cnc\inv bMb\inv=cnaMa\inv c\inv\ne ckaMa\inv c\inv=ckc\inv bMb\inv,
\]
the elements $cnc\inv$ and $ckc\inv$ of $M$ are in distinct cosets of $bMb\inv$,
so $P(b,cnc\inv)$ and $P(b,ckc\inv)$ are distinct terms in the expansion
\[
P(e,e)=\sum_{mbMb\inv\in M/bMb\inv}P(e,m),
\]
and are therefore orthogonal, as desired.

Finally, for \eqref{s pn} we have
\begin{align*}
S_bP(a,c\inv mc)S_b^*
&=S_bS_c^*P(ca,m)S_cS_b^*
\\&=S_c^*P(bca,bmb\inv)S_c
\\&=P(ba,bc\inv mcb\inv).
\qedhere
\end{align*}
\end{proof}

\begin{lem}\label{PN to PNH}
Every PN-family can be uniquely extended to a PNH-family.
\end{lem}

\begin{proof}
Let $h\in H$ and $n\in N$. We want to define $P(h,n)$ as follows: choose $b\in H^+$ such that $hb\in H^+$, and put
\begin{equation}\label{define Ph}
P(h,n)=\sum_{mbMb\inv\in M/bMb\inv}P(hb,nhmh\inv).
\end{equation}
We must show that this is well-defined, and since $H$ is directed by $h\le k$ if and only if $k\in hH^+$ it suffices to show that if $c\in H^+$ then
\[
\sum_{mbMb\inv\in M/bMb\inv}P(hb,nhmh\inv)
=\sum_{\substack{kbcMc\inv b\inv\\\in M/bcMc\inv b\inv}}P(hbc,nhkh\inv).
\]
For each $m\in M$, by relation \eqref{pn consistent} we have
\[
P(hb,nhmh\inv)
=\sum_{lcMc\inv\in M/cMc\inv}P(hbc,nhmblb\inv h\inv),
\]
and so we do have
\begin{align*}
&\sum_{mbMb\inv\in M/bMb\inv}P(hb,nhmh\inv)
\\&\quad=\sum_{\substack{mbMb\inv\\\in M/bMb\inv}}
\sum_{\substack{lcMc\inv\\\in M/cMc\inv}}
P(hbc,nhmblb\inv h\inv),
\end{align*}
which suffices because
as $m$ and $l$ run through complete sets of representatives of cosets of $bMb\inv$ and $cMc\inv$, respectively, the products $mblb\inv$ run through a complete set of representatives of cosets of $bcMc\inv b\inv$.

Thus we have defined projections $\{P(h,n):h\in H,n\in N\}$, and we turn to the relations:
for \eqref{h orthogonal}, let $h\in H$ and $n,k\in N$, and assume that $nhMh\inv\ne khMh\inv$.
Then, since $bMb\inv\subset M$, for all $m,l\in M$ we have
\[
nhmbMb\inv h\inv \ne khlbMb\inv h\inv,
\]
and therefore
\[
P(hb,nhmh\inv)\quad\perp\quad P(hb,khlh\inv).
\]
Summing over $m$ and $l$ (using \eqref{define Ph}, we get
\[
P(h,n)\quad\perp\quad P(h,k).
\]

For the relation \eqref{pnh consistent}, we must now show that for $h\in H$, $n\in N$, and arbitrary $b\in H^+$ we have
\begin{equation}\label{C H}
P(h,n)=\sum_{mbMb\inv}P(hb,nhmh\inv).
\end{equation}
This of course looks very much like the definition of the left-hand side, except that we are not assuming that $hb\in H^+$.
Nevertheless, \eqref{C H} can be proved using the same strategy as our proof that $P(h,n)$ is well-defined: choose $c\in H^+$ such that $hbc\in H^+$, and compute that
\begin{align*}
&\sum_{\substack{mbMb\inv\\\in M/bMb\inv}}\sum_{\substack{lcMc\inv\\\in M/cMc\inv}}
P(hbc,nhmblb\inv h\inv)
\\&\quad=\sum_{\substack{kbcMc\inv b\inv\\\in M/bcMc\inv b\inv}}P(hbc,nhkh\inv)
\end{align*}
as before.

The uniqueness of the family $\{P(h,n):h\in H,n\in N\}$ follows from relations \eqref{h orthogonal}--\eqref{pnh consistent}.
\end{proof}

Recall from \lemref{PN in D} the PN-family $\{p(a,n):a\in H^+,n\in N\}$ in $D$. We now show that $D$ is a universal $C^*$-algebra for relations  \eqref{orthogonal}--\eqref{pn consistent}:

\begin{prop}\label{D}
Assuming \hypref{abelian hyp}, 
for every PN-family $\{P(a,n):a\in H^+,n\in N\}$ in a $C^*$-algebra $B$
there is a unique homomorphism $\pi:D\to B$ such that
\begin{align*}
\pi(p(a,n))&=P(a,n)&&\text{for all }a\in H^+,n\in N.
\end{align*}
\end{prop}

\begin{proof}
$\pi$ will be unique if it exists, as the $p(a,n)$ generate $D$.
Since
$D_0$ is the union of the upward-directed family of $*$-subalgebras
\[
D_a:=\spn\{p(a,n):n\in N\}
\]
indexed by $a\in H^+$,
to construct a $*$-homomorphism $\pi:D_0\to B$
it is enough 
by \lemref{enveloping D}
to find $*$-homomorphisms $\pi_a:D_a\to B$ such that
\begin{equation}\label{consistent}
\pi_{ab}|_{D_a}=\pi_a.
\end{equation}
For each $a\in H^+$ we want to define $\pi_a$ such that
\[
\pi_a(p(a,n))=P(a,n)\midtext{for}n\in N.
\]
Once we verify that this is well-defined, it will of course uniquely extend to a linear map $\pi_a:D_a\to B$.
To see that it is well-defined, suppose that we have $a,b\in H^+$ and $n,k\in N$ with $p(a,n)=p(b,k)$, i.e., 
\[
naMa\inv=kbMb\inv.
\]
Then 
$aMa\inv=bMb\inv$
and
$k\inv n\in aMa\inv$.
Put
$c=a\inv b$
and
$m=a\inv n\inv ka$.
Then $c$ normalizes $M$ and $m\in M$, so by the relation \eqref{pn consistent} we have
\[
P(a,n)=P(ac,nama\inv)=P(b,k).
\]

The relation \eqref{orthogonal} immediately implies that 
the linear map $\pi_a$ is
a $*$-homomorphism of $D_a$,
and then the relation \eqref{pn consistent} implies the consistency condition \eqref{consistent}.
\end{proof}

\begin{lem}\label{a to h}
The unique extension, guaranteed by \lemref{PN to PNH}, of the PN-family $p(a,n)=\Chi_{naMa\inv}$ in $D$ to a PNH-family is given by $p(h,n)=\Chi_{nhMh\inv}$.
Also, the unique homomorphism $\pi:D\to B$, guaranteed by \propref{D}, determined by any PN-family in $B$ takes $p(h,n)$ to the unique extension $P(h,n)$.
\end{lem}

\begin{proof}
This is routine.
\end{proof}

\begin{lem}\label{HN covariant}
Let $B$ be a unital $C^*$-algebra,
and suppose we have unitaries $\{S_a:a\in H^+\}$ and projections $\{P(a,m):a\in H^+,m\in M\}$ in $B$ satisfying \eqref{s rep}, \eqref{s p}, and \eqref{p consistent}.
Extend $P$ uniquely to a PN-family $\{P(a,n):a\in H^+,n\in N\}$ using \lemref{SP to PN},
and further extend uniquely to a PNH-family $\{P(h,n):h\in H,n\in N\}$ using \lemref{PN to PNH}.
Also let $\pi:D\to B$ be the unique $*$-homomorphism guaranteed by \propref{D}.
Then $(\pi,S)$ is a covariant representation of $(D,H)$ in $B$.
\end{lem}

\begin{proof}
By density, covariance is equivalent to the relation
\[
S_hP(a,n)S_h^*=P(ha,hnh\inv)\midtext{for}h\in H,a\in H^+,n\in N,
\]
which, because $H=H^+(H^+)\inv$, follows from \eqref{s pn} and the following computation: if $a,b\in H^+$ and $n\in N$ then
\begin{align*}
S_b^*P(a,n)S_b
&=\sum_{mbMb\inv\in M/bMb\inv}S_b^*P(ab,nama\inv)S_b
\\&=\sum_{mbMb\inv\in M/bMb\inv}P(b\inv ab,b\inv nbb\inv ama\inv b)
\\&=P(b\inv a,b\inv nb).
\qedhere
\end{align*}
\end{proof}

\begin{lem}\label{M to N}
Let $S$ and $U$ be unitary representations of $H$ and $M$, respectively, in a unital $C^*$-algebra $B$, and suppose relation \eqref{s u} holds.
Then $U$ extends uniquely to a unitary representation of $N$ in $B$ such that
\begin{equation}\label{sun}
S_hU(n)S_h^*=U(hnh\inv)\midtext{for}h\in H,n\in N.
\end{equation}
\end{lem}

\begin{proof}
Since $N$ is the union of the upward-directed family of subgroups $a\inv Ma$ for $a\in H^+$, to define a representation of $N$ it suffices to find representations $U_a:a\inv Ma\to\BB(X)$ such that
\[
U_{ab}|_{a\inv Ma}=U_a.
\]
It is routine to verify that
\[
U_a(a\inv ma):=S_a^*U(m)S_a
\]
does the job, and then \eqref{sun} follows from the following computation: for $a,b,c\in H^+$ and $m\in M$ we have
\begin{align*}
S_{ab\inv}U(c\inv mc)S_{ab\inv}^*
&=S_{ab\inv}S_c^*U(m)S_cS_{ab\inv}^*
\\&=S_{cb}^*S_aU(m)S_a^*S_{cb}
\\&=S_{cb}^*U(ama\inv)S_{cb}
\\&=U(ab\inv c\inv mcba\inv).
\end{align*}
Finally, the uniqueness follows quickly from the relation \eqref{sun}.
\end{proof}

\subsection*{Proofs of the theorems}

In each of Theorems~\ref{pCp} and \ref{pAp} we are given an isometric representation $S$ of $H^+$ in a unital $C^*$-algebra $B$.
We first show how to use the dilation technique originally developed  for the abelian case (independently) by 
Brehmer \cite{B} and 
It\^o \cite{Ito}. 
However, we follow closely the approach of Douglas in \cite{RGD}. 
The result has been generalised to nonabelian groups in \cite{PRIsometries, LRmultiplier, lac:corner}, but since we have no application in mind for the nonabelian case, we have kept everything abelian in order to simplify the arguments.

First represent $B$ faithfully and nondegenerately on a Hilbert space $X$.
\cite[Theorem~1]{RGD} constructs a Hilbert space $\wilde X$, a homomorphism $\wilde S$ from $H^+$ to the unitary group of $\wilde X$, and an isometric embedding $T:X\to\wilde X$ such that
\begin{align*}
\wilde S_aT&=TS_a\midtext{for}a\in H^+\\
\wilde X&=\bar{\bigcup_{a\in H^+}\wilde S_a^*T(X)}.
\end{align*}
Actually, we will not need Douglas' construction of $\wilde X$; in fact, we only need the above properties.

\begin{proof}[Proof of \thmref{pCp}]
As above, dilate $S$ to a homomorphism $\wilde S$ from $H^+$ to the unitary group of $\wilde X$,
together with an isometry $T:X\to \wilde X$ intertwining $S$ and $\wilde S$.
Apply \lemref{H^+ to H} to extend $\wilde S$ to a unitary representation of $H$ on $\wilde X$.

We now use the dilation method to construct projections $\{\wilde P(a,m):a\in H^+,m\in M\}$ on $\wilde X$
such that $\wilde S$, $\wilde P$ satisfy \eqref{s p}--\eqref{p consistent}
and
\begin{equation}\label{TP}
TP(a,m)=\wilde P(a,m)T\midtext{for}a\in H^+,m\in M.
\end{equation}

Since $\wilde X$ is the closure of the union of the upward-directed family of closed subspaces $\wilde X_b=\wilde S_b^*T(X)$ indexed by $b\in H^+$,
to construct a projection $\wilde P(a,m)$ on $\wilde X$ it suffices to find projections $\wilde P_b(a,m)$ on the subspaces $\wilde X_b$ such that
\begin{equation}\label{P consistent}
\wilde P_{bc}(a,m)|_{\wilde X_b}=\wilde P_b(a,m).
\end{equation}
The appropriate definition is dictated by our desire for covariance: for $\xi\in X$ define
\begin{equation}\label{define wp}
\wilde P_b(a,m)\wilde S_b^*T\xi
=\wilde S_b^*TP(ba,bmb\inv)\xi.
\end{equation}
It follows quickly that $\wilde P_b(a,m)$ is a projection on $\wilde X_b$.

To see the consistency relation \eqref{P consistent}, let $\xi\in X$ and compute:
\begin{align*}
\wilde P_b(a,m)\wilde S_b^*T\xi
&=\wilde S_b^*TP(ba,bmb\inv)\xi
\\&=\wilde S_{cb}^*TS_cP(ba,bmb\inv)\xi
\\&=\wilde S_{cb}^*TP(cba,cbmb\inv c\inv)S_c\xi
\\&=\wilde P_{cb}(a,m)\wilde S_{cb}^*TS_c\xi
\\&=\wilde P_{cb}(a,m)\wilde S_b^*T\xi.
\end{align*}
Thus there is a unique projection $\wilde P(a,m)$ on $\wilde X$ extending the operators $\wilde P_b(a,m)$ on the subspaces $\wilde X_b$.
It suffices to verify relations \eqref{s p}--\eqref{p consistent} after post-multiplying with $\wilde S_c^*T$ for any $c\in H^+$:
we have
\begin{align*}
\wilde S_b\wilde P(a,m)\wilde S_b^*\wilde S_c^*T
&=\wilde S_b\wilde S_{cb}^*TP(cba,cbmb\inv c\inv)
\\&=\wilde P(ba,bmb\inv)\wilde S_c^*T,
\end{align*}
and
\begin{align*}
\wilde P(a,k)\wilde S_c^*T
&=\wilde S_c^*TP(ca,ckc\inv)
\\&=\wilde S_c^*T\sum_{mbMb\inv\in M/bMb\inv}P(cab,ckc\inv cama\inv c\inv)
\\&=\sum_{mbMb\inv\in M/bMb\inv}\wilde P(ab,kama\inv)\wilde S_c^*T.
\end{align*}
Finally, \eqref{TP} follows from \eqref{define wp} with $b=e$.

Next, apply Lemmas~\ref{SP to PN}, \eqref{PN to PNH}, \eqref{HN covariant}, and \propref{D} to extend $\wilde P$ to a PNH-family $\{\wilde P(h,n):h\in H,n\in N\}$ and get a representation $\pi$ of $D$ on $\wilde X$.
The covariant representation $(\pi,\wilde S)$ of $(D,H)$ gives a representation $\wilde \Pi:=\pi\times \wilde S$ of $C=D\rtimes H$ on $\wilde X$.
Define $\Pi:pCp\to \BB(X)$ by
\[
\Pi(z)=T^*\wilde\Pi(z)T.
\]
Since $\wilde\Pi$ takes $p=\Chi_M$, the unit of the corner $pCp$, to $\wilde P(e,e)=TT^*$,
it takes $pCp$ to the corner $TT^*\BB(\wilde X)TT^*$, and it follows that $\Pi$ is a representation.
For $a\in H^+$ and $m\in M$ we have
\begin{align}
\Pi(s_a)&=T^*\wilde S_aT=S_a\label{s}
\\
\Pi(p(a,m))&=T^*\wilde P(a,m)T=P(a,m).\label{p}
\end{align}
Since $\{s_a,p(a,m):a\in H^+,m\in M\}$ generates $pCp$,
it follows that $\Pi$ maps $pCp$ into the $C^*$-algebra $B$,
and is moreover the unique homomorphism satisfying \eqref{s}--\eqref{p}.
\end{proof}

\begin{proof}[Proof of \thmref{pAp}]
Apply \lemref{SU to SP} to get an SP-family $\{S_a,P(a,m):a\in H^+,m\in M\}$,
and then apply \thmref{pCp} and its proof to get a covariant representation $(\pi,\wilde S)$ of $(D,H)$ on a Hilbert space $\wilde X$,
together with an isometry $T:X\to \wilde X$
and a PHN-family $\{\wilde P(h,n):h\in H,n\in N\}$.

We now construct
a unique unitary representation $\wilde U$ of $M$ on $\wilde X$ such that
\begin{align}
\label{UT}
\wilde U(m)T&=TU(m)
\\
\label{SU}
\wilde S_a\wilde U(m)\wilde S_a^*&=\wilde U(ama\inv)
\end{align}
for $a\in H^+,m\in M$.
To construct
a unitary operator $\wilde U(m)$ on $\wilde X$ it is enough to find a unitary $\wilde U_a(m)$
on each subspace $\wilde X_a=\wilde S_a^*T(X)$
such that
\begin{equation}\label{U consistent}
\wilde U_{ab}(m)|_{\wilde X_a}=\wilde U_a(m).
\end{equation}
For $\xi\in X$ define
\[
\wilde U_a(m)\wilde S_a^*T\xi=\wilde S_a^*TU(ama\inv)\xi.
\]
Since $\wilde S_a^*$ restricts to an isometric isomorphism of the subspace $T(X)$ onto the subspace $\wilde X_a$, with inverse the corresponding restriction of $\wilde S_a$, it follows that $\wilde U_a(m):\wilde X_a\to \wilde X_a$ is unitary.

To see the consistency relation \eqref{U consistent}, let $\xi\in X$ and compute:
\begin{align*}
\wilde U_a(m)\wilde S_a^*T\xi
&=\wilde S_a^*TU(ama\inv)\xi
\\&=\wilde S_{ab}^*TU(abmb\inv a\inv)S_b\xi
\\&=\wilde U_{ab}(m)\wilde S_a^*T\xi.
\end{align*}

Thus there is a unique unitary operator $\wilde U(m)$ on $\wilde X$ extending the operators $\wilde U_a(m)$ on the subspaces $\wilde S_a^*T(X)$.
It follows quickly that this gives a unitary representation $\wilde U$ of $M$.
\eqref{UT} follows from
\[
\wilde U(m)T=\wilde U(m)\wilde S_e^*T=\wilde S_e^*TU(m)=TU(m).
\]
For \eqref{SU} it suffices to observe that for all $b\in H^+$ we have
\begin{align*}
\wilde S_a\wilde U(m)\wilde S_a^*\wilde S_b^*T
&=\wilde S_a\wilde S_{ba}^*TU(bama\inv b\inv)
\\&=\wilde S_b^*TU(bama\inv b\inv)
\\&=\wilde U(ama\inv)\wilde S_b^*T.
\end{align*}

Now apply \lemref{M to N} to extend $\wilde U$ to a unitary representation of $N$ on $\wilde X$ such that
\[
\wilde S_h\wilde U(n)\wilde S_h^*=\wilde U(hnh\inv)\midtext{for}h\in H,n\in N.
\]
Then it is obvious that
\[
hn\mapsto \wilde S_h\wilde U(n)
\]
gives a unitary representation of $G$.

We will show that $(\pi,\wilde S\wilde U)$ is a covariant representation of $(D,G)$:
since $(\pi,\wilde S)$ is a covariant representation of $(D,H)$, by density it is enough to show
\[
\wilde U(k)\wilde P(a,n)\wilde U(k)^*=\wilde P(a,kn)\midtext{for}k\in N,
\]
which follows immediately from the constructions.

The remainder of the proof follows the same lines as \thmref{pCp}.
\end{proof}

As in \cite[Remark~3.7]{CLintegral}, we deduce:

\begin{cor}\label{pAp unique}
$pAp$ is, up to isomorphism, the unique $C^*$-algebra generated by an SU-family.
\end{cor}

\begin{proof}
This follows immediately from \thmref{pAp} and \corref{enveloping corner}.
\end{proof}

\section{The Cuntz-Li ring $C^*$-algebras}
\label{ring}

As in \secref{integral}, let $R$ be an integral domain that is not a field, and keep the other assumptions and notation from \secref{integral}.
We will apply our general theory to give an independent proof of the following result of Cuntz and Li:

\begin{thm}
\label{Cuntz Li}
The $C^*$-algebra $\fA[R]$ is 
a \(full\) corner of the crossed product $D\rtimes_\alpha G$,
and hence is
simple and purely infinite.
\end{thm}

Note that in \cite{CLintegral}
Cuntz and Li denote our group $G=N\rtimes H$ by $P_{Q(R)}$.
Recall that Cuntz and Li 
studied $\fA[R]$ using
auxiliary objects 
$\what R$, $\RR$, and
$\fD(R)$.

Our $\bar M$ coincides with $\what R$, since
\cite{CLintegral} constructs the topological group $\what R$ as the completion of $R$.
Further, our $\bar N$ can be identified with $\RR$,
since \cite{CLintegral} constructs $\RR$ as an inductive limit of the maps on $\what R$ given by multiplication by elements of the directed multiplicative semigroup $R$.
Thus, by \cite[Observation~4.3]{CLintegral}, our $D$ is isomorphic to $\fD(R)$.

Strictly speaking, to make the connection with \cite{CLintegral} valid, we must note that
this time we can get away with a slightly smaller family
\[
\UU=\{aR:a\in R\minus\{0\}\},
\]
rather than letting $a$ run through all of $Q^\times$ as in the previous sections.
Then from the properties of (non-field) integral domains, 
conditions~\eqref{finite} and~\eqref{effective} 
are satisfied.
That using this smaller version of $\UU$ will not change our results follows from:

\begin{lem}
\label{smaller}
$\UU$ is cofinal in $\UU'=\{aR:a\in Q\minus\{0\}\}$.
\end{lem}

\begin{proof} Clearly $\UU\subset\UU'$.
If
\[
a=\frac st\midtext{with}s,t\in R,t\ne 0,
\]
then
\[
aR=\frac st R\supset sR
\]
and it follows that $\UU$ is cofinal in $\UU'$ and that they define the same group topology on $Q$.
\end{proof}

\begin{proof}
[Proof of \thmref{Cuntz Li}]
\corref{pAp unique} 
shows that the corner $p(D\rtimes_\alpha G)p$ is the unique $C^*$-algebra for the same generators and relations as \cite[Definition~2.1]{CLintegral} use to define $\fA[R]$, so we have $\fA[R]\cong p(D\rtimes_\alpha G)p$.
By \thmref{domain theorem}
the crossed product $D\rtimes_\alpha G$ is simple and purely infinite, and therefore so is the (full) corner $p(D\rtimes_\alpha G)p$.
\end{proof}



\providecommand{\bysame}{\leavevmode\hbox to3em{\hrulefill}\thinspace}
\providecommand{\MR}{\relax\ifhmode\unskip\space\fi MR }
\providecommand{\MRhref}[2]{%
  \href{http://www.ams.org/mathscinet-getitem?mr=#1}{#2}
}
\providecommand{\href}[2]{#2}

\end{document}